\documentclass[11pt]{article}

\usepackage{amsmath}
\usepackage{amssymb}

\usepackage{amsthm}     
\theoremstyle{plain}     
\newtheorem{theorem}{Theorem}

\newtheorem{corollary}[theorem]{Corollary}

\newtheorem{lemma}[theorem]{Lemma}

\theoremstyle{definition} 

\theoremstyle{remark} 
\newtheorem{remark}{Remark}

\renewenvironment{proof}[1][Proof]{\noindent\textbf{#1.} }{\ $\Box$\medskip}
\usepackage{natbib}

\title{A Functional Equation of Tail-balance \\ for Continuous Signals in the
Condorcet \\ Jury Theorem}

\author{Steve Alpern\footnote{Warwick Business School, University of Warwick, Scarman Rd, Coventry CV4 7AL;
Steve.Alpern@wbs.ac.uk} \and Bo Chen\footnote{Warwick Business School, University of Warwick,
Scarman Rd, Coventry CV4 7AL; Bo.Chen@wbs.ac.uk} \and Adam J. Ostaszewski\footnote{Mathematics
Department, London School of Economics, Houghton Street, London WC2A 2AE;
A.J.Ostaszewski@lse.ac.uk} }

\date{November 22, 2019}

\begin{document}

\maketitle

\begin{abstract}
Consider an odd-sized jury, which determines a majority verdict between two equiprobable states
of Nature. If each juror independently receives a binary signal identifying the correct state
with identical probability $p$, then the probability of a correct verdict tends to one as the
jury size tends to infinity (Condorcet, 1785). Recently, \citet{AlpC2,AlpC1} developed a model
where jurors sequentially receive signals from an interval according to a distribution, which
depends on the state of Nature and on the juror's ``ability'', and vote sequentially. This paper
shows that to mimic Condorcet's binary signal, such a distribution must satisfy a functional
equation related to tail-balance, that is, to the ratio $\alpha(t)$ of the probability that a
mean-zero random variable satisfies $X$ $>t$ given that $|X|>t$. In particular, we show that
under natural symmetry assumptions the tail-balances $\alpha(t)$ uniquely determine the
distribution.
\end{abstract}

\section{Introduction}
\label{sec:Introduction}

This paper studies a functional equation arising from an extension of the celebrated Condorcet Jury
Theorem. In Condorcet's model, an odd-sized jury must decide whether Nature is in one of two
equiprobable states of Nature, $ A $ or $B$. Each juror receives an independent binary signal (for
$A$ or for $B $) which is correct with the same probability $p>1/2$. \citet{Con} [or see
\citet[][Ch.~XVII]{Tod} for a textbook discussion] showed that when jurors vote simultaneously
according to their signal, the probability of a correct verdict tends to $1$ as the number of
jurors tends to infinity. Recently, \citet{AlpC2,AlpC1} considered a related sequential voting
model, where jurors receive signals $S$ in the interval $[-1, +1] $ rather than binary signals. Low
signals indicate $B$ and high signals indicate $A$. The strength of this ``indication'' depends on
the ``ability'' of the juror, a number $a$ between $0$ and $1$, which is a proxy for Condorcet's
$p$ that tends from $1/2$ to $1$. When deciding how to vote, each juror notes the previous voting,
the abilities of the previous jurors, his own signal $S$ and his own ability. This is sufficient to
determine which alternative he views as being more likely. The mechanism that underlies this
determination is the common knowledge of the distributions by which a signal is given as private
information to each juror, depending on his ability and the state of Nature. It is not relevant to
the discussions of this paper, but we mention that one of the main results of Alpern and Chen is
that, given three jurors of fixed abilities, their majority verdict is most likely to be correct
when they vote in the following order: middle-ability juror first, highest-ability juror next, and
finally the lowest-ability juror.

The cumulative distribution formula of signals $S$ on $\left[ -1,+1\right] $ that a juror of
ability $a$ receives in the Alpern-Chen jury model is given by
\begin{align}
F_{a}(t) := \mathbb{P}_{a}[S\leq t \,|\,  A] =(t+1)(at-a+2)/4, & \text{ if Nature is $A$;}\label{Fa} \\
G_{a}(t) := \mathbb{P}_{a}[S\leq t \,|\,  B]= ( t+1)( a-at+2)/4, &\text{ if Nature is $B$.}  \label{Ga}
\end{align}
These were selected as the simplest family of distributions arising from linear densities in which
steepness indicates ability: for the signal distributions following $A$ or $B$, a high signal in
$[-1,+1]$ was to indicate $A$ as more likely, and a low signal to indicate $B$ as more likely. For
this reason an increasing density function was selected to follow $A$ and a decreasing one to
follow $B$. The simplest increasing and decreasing densities on $[-1,+1]$ are the linear functions
(taken to mean ``affine'') that go through $(0,1/2)$. This gave the signal distribution of
Alpern-Chen model. It emerges from results below that these are uniquely determined from requiring
the tail-balance to be linear.

The relation between the two distributions is based on the following assumption of signal symmetry:
\[
\mathbb{P}_{a}[S\leq t\,|\, A]=\mathbb{P}_{a}[-S\leq t\,|\, B].
\]
For a continuous distribution it follows that
\begin{align*}
G_{a}(t) &=\mathbb{P}_{a}[S\leq t\,|\,  B]=\mathbb{P}_{a}[-S\geq -t \,|\,  B]
 =1-\mathbb{P}_{a}[-S\leq -t \,|\,  B] \\
        &=1-\mathbb{P}_{a}[S\leq -t \,|\,  A]=1-F_{a}(-t),
\end{align*}
and so
\[
G_{a}(t) =1-F_{a}(-t).  
\]
This confirms that, for a juror of any ability $a$, the probability of receiving a signal less than
$t$ when Nature is $B$ is the same as receiving a signal larger than $t$ when Nature is $A$.

The jurors use their private information (signal) $S$ in $[0,1]$ to calculate the conditional
probabilities of $A$ and $B$ by considering the relative likelihood that their signal came from the
distribution $F$ or the distribution $G$.

We now wish to relate this continuous signal model to the binary model of Condorcet. We want our
notion of ability $a$ to be a proxy for Condorcet's probability $p$. Since his $p$ runs from
$p_0=1/2$ (a useless signal) to $p_1=1$ (a certain signal) and our ability $a$ runs from $0$ (no
ability, useless signals) to $1$ (highest ability), for fixed $t$ the conditional probability is a
linear transformation, and so for some coefficient $b=b(t)$ we wish to have
\[
\mathbb{P}_{a}[A\,|\, S\geq t] =\frac{1}{2}+b(t)\, a,
\]
as the left-hand side is his $p$ if his signals are restricted to $-1$ (for $B$) and $+1$ (for
$A$), taking any $t$ in $( 0,1)$. We want this conditional probability to be $1$ when $t=+1$
(highest signal) and $a=1$ (highest ability), so this gives $b(1) =1/2$. When $t=-1$ the condition
$S\geq t$ gives no new information for any $a$, so the left-hand side should be the \emph{a priori}
probability of $A$, $\mathbb{P}_a[A]$, which is $1/2$. Hence $b(-1) =0$ and, taking $b$ to be
linear, we get $b(t) =(t+1)/4$, or
\[
\mathbb{P}_{a}[A\,|\, S\geq t] =\frac{1}{2}+\frac{t+1}{4}\,a.
\]
Putting $H_{a}( t) :=\mathbb{P}_{a}[ S\leq t\,|\, A] $, so that $\mathbb{P}_{a}[S\leq t\,|\,
B]=1-H_{a}(-t)$ as above, by Bayes Law and since $\mathbb{P}[B] =\mathbb{P}[A] $, we have
\begin{eqnarray*}
\mathbb{P}_{a}[ A\,|\, S\geq t]
&=&\frac{\mathbb{P}_{a}[ S\geq t\,|\, A] \mathbb{P}[ A] }{\mathbb{P}_{a}[ S\geq t\,|\, A]
\mathbb{P} [A] +\mathbb{P}_{a}[ S\geq t\,|\, B] \mathbb{P}[ B] } \\
&=&\frac{1-H_{a}( t) }{( 1-H_{a}( t) )+( 1-( 1-H_{a}( -t) ) ) } \\
&=&\frac{1-H_{a}( t) }{1-H_{a}( t) +H_{a}(-t) }.
\end{eqnarray*}%
The last term, comparing the right tail against the tail sum, is known as the \textit{tail-balance
ratio} \citep[\S8.3]{BinGT}. Its asymptotic behaviour and the regular variation of the tail-sum are
particularly relevant to the Domains of Attraction Theorem of probability theory
\citep[Theorem~8.3.1]{BinGT}. The theorem is concerned with stable laws (stable under addition: the
sum of two independent random variables with that law has, to within scale and centering, the same
law), and identifies those that arise as limits in distribution of appropriately scaled and
centered random walks.

In summary we seek a family, indexed by the ability $a$, of signal distributions $H_{a}\left(
t\right) $ on $\left[ -1,+1\right] $, which correspond to state of Nature $A$ while $1-H_{a}\left(
-t\right) $ correspond to state of Nature $B$, such that by Bayes Law
\[
\mathbb{P}_{a}[ A\,|\, S\geq t] =\frac{1-H_{a}(t)}{1-H_{a}\left( t\right)
+H_{a}(-t) }=\frac{1}{2}+\frac{t+1}{4}~a.
\]
The main technical result of the paper is the following.

\begin{lemma}\label{lem:main}
The unique solution for the c.d.f.\ $H_{a}( t)$ on $[ -1,+1 ] $ to the following functional
equations for $0\leq a\leq 1$:
\[
\frac{1-H_{a}( t) }{1-H_{a}( t) +H_{a}( -t)}=\frac{1}{2}+\frac{t+1}{4}\,a,\ -1\leq t\leq +1
\]
is given by
\[
H_{a}\left( t\right) =F_{a}( t) =( t+1) (at-a+2) /4.
\]
\end{lemma}

The standard text-book treatment of functional equations is \citet{AczD}, but it is often the case
that particular functional equations arising in applications require individual treatment ---
recent such examples are \citet{ElsBFN} and \citet{KahM}; for applications in probability, see
\citet{Ost}.

We will prove this result in Section~\ref{sec:2}, which thus gives the following consequence for
the signal distribution in the jury problem.

\begin{theorem}
The only c.d.f.\ on the signal space $[-1,+1] $ that makes the conditional probability
$\mathbb{P}_{a}[A\,|\, S\geq t] $ a linear function of the juror's ability $a$ with a slope linear
in $t$ are the Alpern-Chen functions $F_{a}(t)$ and $G_{a}(t)$ given in (\ref{Fa}) and (\ref{Ga}).
\end{theorem}

\section{Tail-balance equation and proof of Lemma~\ref{lem:main}}
\label{sec:2}

The proof of Lemma~\ref{lem:main} in this section is deduced from a more general result concerning
a functional equation of the following type:
\begin{equation}\label{eqn:tail-equation}
  \frac{1-H(t)}{1-H(t)+H(-t)}=\alpha (t),\ t\in [-1,+1),
\end{equation}
with $\alpha (t)$ a strictly monotone function, interpreting the left-hand side to be $1$ for
$t=1$. Of interest here are non-negative increasing functions $H$ with $H(-1)=0$ and $H(1)=1$
representing probability distribution functions, hence the adoption of the name ``tail balance''
(as above). Indeed, with these boundary conditions,
\[
\alpha(-1)=\frac{1}{2},
\]
implying that the left and right tails of $H$ are exactly balanced; furthermore,
\[
\frac{1}{2}<\alpha (t)<1,\ t\in (-1,+1),
\]
since
\[
\frac{1-H(t)}{1-H(t)+H(-t)}=1-\frac{H(-t)}{1-H(t)+H(-t)}.
\]
Note that $H(0)=1-\alpha (0)$. The linear case of the last section is thus
\[
\alpha(t)=\frac{1}{2}+\frac{t+1}{4}\,a,\ 0\le a\le 1.
\]
Turning to a general increasing $\alpha $, we may write
\[
\beta (t):=\frac{\alpha (t)}{1-\alpha (t)},\ t\in [-1,+1).
\]
This is again an increasing function with $\beta (t)>1$ for $t\in (-1,+1)$ and $\lim_{t\uparrow
1}\beta(t)=+\infty$ if $\alpha(1)=1$. Thus in the formula below $ \beta(t)\beta(-t)-1>0$ and
$\lim_{t\uparrow 1}(\beta (t)-1)/(\beta (t)\beta (-t)-1)=1$, since $\beta (-1)=1$.

\begin{theorem}\label{thm:tail-solution}
The tail-balance functional equation \eqref{eqn:tail-equation} has the following unique
non-negative solution:
\begin{equation}
H(t)=\frac{\beta (t)-1}{\beta (t)\beta (-t)-1}=\frac{(2\alpha
(t)-1)(1-\alpha (-t))}{\alpha (t)+\alpha (-t)-1}<1.  \label{Soln Tail}
\end{equation}
In particular, for $\alpha $ linear as in Lemma~\ref{lem:main}, we have
\begin{equation}
H(t)=\frac{( t+1) ( at-a+2)}{4}.  \label{Soln Lin}
\end{equation}
\end{theorem}

\begin{proof}
After some re-arrangement of \eqref{eqn:tail-equation} we have
\begin{equation*}
H(-t)=\frac{1-H(t)}{\beta (t)}.
\end{equation*}%
So%
\begin{equation*}
\beta (-t)H(t)=1-H(-t)=1-\frac{1-H(t)}{\beta (t)}.
\end{equation*}%
Hence%
\begin{equation*}
(\beta(t)\beta (-t)-1)H(t)=\beta (t)-1,
\end{equation*}%
yielding the asserted formula. As for the inequality, we note the equivalence:
\begin{equation*}
\frac{\beta (t)-1}{\beta (t)\beta (-t)-1}<1\Longleftrightarrow \beta
(t)<\beta (t)\beta (-t).
\end{equation*}%
The calculation of the linear case of $\alpha $ is straightforward, and relies on
\begin{equation*}
\alpha (t)+\alpha (-t)-1=a/2,\quad \text{and}\quad 1-\alpha (-t)=(at-a+2)/4.
\end{equation*}
Our theorem follows.
\end{proof}

\begin{remark}
More generally, with $\alpha (t)$ monotone as before and $\mathbb{P}_a[A] =\theta $ with $0<\theta
<1$, so that $\mathbb{P}_a[ B] =1-\theta $, writing the odds $(1-\theta )/\theta $ as $\lambda ,$
the earlier application of Bayes Rule gives for $t\in [-1,+1]$
\begin{equation}\label{Odds}
\mathbb{P}_{a}[ A\,|\, S\geq t] =\frac{1-H_{a}\left( t\right) }{1-H_{a}\left( t\right)
+\lambda H_{a}\left( -t\right) }=\alpha (t).
\end{equation}
Here, as $1+\lambda =1+(1-\theta )/\theta =1/\theta$,
\begin{equation*}
\alpha (-1)=\frac{1}{1+\lambda }=\theta =\mathbb{P}_{a}[ A]
=\mathbb{P}_{a}[ A\,|\, S\geq -1] .
\end{equation*}
\end{remark}

To solve (\ref{Odds}) we apply a similar procedure as in Theorem~\ref{thm:tail-solution} by first
showing the following variant.

\begin{theorem}
The equation for $t\in [-1,+1]$
\begin{equation}
H(-t)=\gamma (t)+\delta (t)H(t)  \label{General}
\end{equation}%
has solution
\begin{equation}
H(t)=\frac{\gamma (t)\delta (-t)+\gamma (-t)}{1-\delta (t)\delta (-t)},
\label{Soln Gen}
\end{equation}%
provided $\delta (t)\delta (-t)\neq 1$.
\end{theorem}

\begin{proof}
Equation (\ref{General}) may be solved by writing
\begin{eqnarray*}
H(t) &=&\gamma (-t)+\delta (-t)H(-t) \\
&=&\gamma (-t)+\delta (-t)(\gamma (t)+\delta (t)H(t)), \\
H(t)(1-\delta (t)\delta (-t)) &=&\gamma (t)\delta (-t)+\gamma (-t),
\end{eqnarray*}%
yielding the claim.
\end{proof}

\begin{corollary}\label{cor:solution}
Equation (\ref{Odds}) has solution for $t\in [-1,+1]$
\begin{equation*}
H(t)=\frac{((\lambda +1)\alpha (t)-1)(1-\alpha (-t))}{\alpha (t)+\alpha
(-t)+(\lambda ^{2}-1)\alpha (t)\alpha (-t)-1}.
\end{equation*}
\end{corollary}

\begin{proof}
Equation (\ref{Odds}) may be rewritten as
\[
(1-\alpha (t))=(1-\alpha (t))H(t)+\lambda \alpha (t)H(-t),\ t\in [-1,+1],
\]
so is of the more general form above with
\begin{equation}
\gamma (t)=\frac{1-\alpha (t)}{\lambda \alpha (t)},\quad \delta (t)=-\gamma
(t).  \label{gamma}
\end{equation}%
For $\gamma (t)=-\delta (t)$ equation (\ref{Soln Gen}) becomes%
\begin{equation}
H(t)=\frac{-\gamma (t)\gamma (-t)+\gamma (-t)}{1-\gamma (t)\gamma (-t)}
=\frac{\gamma (-t)(1-\gamma (t))}{1-\gamma (t)\gamma (-t)}.  \label{Soln Spec}
\end{equation}%
Substitution in (\ref{Soln Spec}) for $\gamma (t)$ from (\ref{gamma}) gives
\begin{align*}
H(t) &=\frac{1-\alpha (-t)}{\lambda \alpha (-t)}\frac{1-\frac{1-\alpha (t)}{\lambda \alpha (t)}}
{1-\frac{1-\alpha (t)}{\lambda \alpha (t)}\frac{1-\alpha(-t)}{\lambda \alpha (-t)}} \\
 &=\frac{1-\alpha (-t)}{\alpha (-t)}\cdot \frac{\lambda \alpha (t)\alpha(-t)-(1-\alpha (t))
  \alpha (-t)}{(\lambda ^{2}\alpha (t)\alpha(-t))-(1-\alpha (t))(1-\alpha (-t))}   \\
 &=\frac{((\lambda +1)\alpha (t)-1((1-\alpha (-t))}{(\lambda ^{2}\alpha(t)\alpha (-t))
  -(1-\alpha (t))(1-\alpha (-t))}   \\
 &=\frac{((\lambda +1)\alpha (t)-1)(1-\alpha (-t))}{\alpha (t)+\alpha(-t)
  +(\lambda ^{2}-1)\alpha (t)\alpha (-t)-1}.  
\end{align*}%
Note that, as $\alpha $ is increasing, $\alpha (t)>\theta $ for $t\in (-1,+1] $ and then $(1-\alpha
(t))/\alpha (t)<(1-\theta )/\theta =\lambda$. So, for $t\in \lbrack -1,+1]$, the denominator in
(\ref{Soln Tail}) is non-zero, as
\begin{equation*}
\frac{1-\alpha (t)}{\alpha (t)}\cdot \frac{1-\alpha (-t)}{\alpha (-t)}<\lambda
^{2}.
\end{equation*}
\end{proof}

\begin{remark}
When $\lambda =1$ the above corollary yields (\ref{Soln Tail}) of Theorem~\ref{thm:tail-solution}.
\end{remark}

\begin{theorem}
The linear case of the general odds tail-balance equation (\ref{Odds}), i.e., with
\[
\alpha (t)=\theta +(t+1)(1-\theta )a/2,
\]
has solution
\begin{align}
H(t) &=\frac{(1+t)(at-a+2)a/4}{-a^{2}/4+a^{2}t^{2}/4+(a+\lambda a^{2}/4
        -\lambda a^{2}t^{2}/4)}  \label{Soln Odds} \\
    & =\left\{\begin{array}{ll}
        (1+t)(at-a+2)/(4-a+at^{2}), & \textrm{as }\lambda \rightarrow 0, \\
        (1+t)(at-a+2)a/4, & \textrm{if } \lambda =1, \\
        (at-a+2)/(\lambda a(1-t))=2/(\lambda a(1-t))-1/\lambda , & \textrm{as }
        \lambda \rightarrow \infty.
\end{array}%
\right. \nonumber
\end{align}
\end{theorem}

\begin{proof}
In the general state of Nature case (\ref{Odds}), specializing to the linear case and repeating the
argument in Section~\ref{sec:Introduction} does indeed give
\[
\alpha (t)=\theta +(t+1)(1-\theta )a/2,
\]
as then $\alpha (-1)=\theta $ and, for $a=1$, $\alpha (1)=\theta +(1-\theta )=1$. Noting that
\[
1-\alpha (-t) =1-\theta -\frac{1}{2}(1-t)(1-\theta )a=\frac{1}{2}(1-\theta)(at-a+2)
\]
and
\[
\alpha (t)+\alpha (-t)-1 = 2\theta +(1-\theta )a-1=\theta +(1-\theta )(a-1),
\]
and, writing $\alpha (t)=B+At$ for convenience, from Corollary~\ref{cor:solution} we derive
\begin{eqnarray*}
H(t) &=&\frac{((\lambda +1)(B+At)-1)(1-\theta )(at-a+2)/2}{\theta +(1-\theta
)(a-1)+(\lambda ^{2}-1)(B^{2}-A^{2}t^{2})} \\
&=&\frac{((B+At)/\theta -1)\lambda (at-a+2)/2}{1+\lambda (a-1)+(\lambda
^{2}-1)(B^{2}-A^{2}t^{2})/\theta } \\
&=&\frac{\lambda ^{2}(1+t)(at-a+2)a/4}{1+\lambda (a-1)+(\lambda
-1)(B^{2}-A^{2}t^{2})/\theta ^{2}},
\end{eqnarray*}%
as $1+\lambda =1/\theta $ (as above), $B/\theta =1+\lambda a/2$ and $A/\theta =\lambda a/2$.

Finally, writing $\mu =1/\lambda$, we obtain
\begin{equation*}
H(t)=\frac{(1+t)(at-a+2)a/4}{\mu ^{2}+\mu (a-1)+(1/\mu -1)((\mu
+a/2)^{2}-a^{2}t^{2}/4)},
\end{equation*}%
and here the denominator is
\begin{equation*}
D(t):=-a^{2}/4+a^{2}t^{2}/4+(a+\lambda a^{2}/4-\lambda a^{2}t^{2}/4),
\end{equation*}%
so that $D(t)\sim \lambda a^{2}(1-t^{2})/4$ as $\lambda \rightarrow \infty$ and
\[
\lim_{\lambda \rightarrow 0}D(t;\mu ) =\lim_{\lambda \rightarrow
0}(4a-a^{2}+a^{2}t^{2})/4.
\]
\end{proof}

\begin{remark}
When $\lambda =1$ we retrieve from (\ref{Soln Odds}) the formula $(1+t)(at-a+2)/4$ as in (\ref{Soln
Lin}).
\end{remark}

\section{An alternative proof of Lemma~\ref{lem:main}}

In this section we give, as an alternative to Theorem~\ref{thm:tail-solution}, a direct proof of
Lemma~\ref{lem:main}, as it is of independent interest.

Consider the following equation:
\begin{equation}
\frac{1-H_{a}(t)}{1-H_{a}(t)+H_{a}(-t)}=\frac{2+(1+t)a}{4}.
\label{eqn:functional_equation}
\end{equation}%
We are to show that $H_{a}(t)=(1+t)(2+at-a)/4$. Introduce two functions as follows:
\begin{align*}
f_{a}(t)& :=H_{a}(t)-H_{a}(-t), \\
g_{a}(t)& :=H_{a}(t)+H_{a}(-t).
\end{align*}%
Then \eqref{eqn:functional_equation} can be re-written as
\begin{equation}
H_{a}(t)=1-\frac{2+(1+t)a}{4}\left( 1-f_{a}(t)\right) .
\label{eqn:functional_equation1}
\end{equation}%
Similarly, replacing $t$ by $-t$ in \eqref{eqn:functional_equation} leads to
\begin{equation}
H_{a}(-t)=1-\frac{2+(1-t)a}{4}\left( 1+f_{a}(t)\right) .
\label{eqn:functional_equation2}
\end{equation}%
Subtraction of \eqref{eqn:functional_equation2} from %
\eqref{eqn:functional_equation1} gives
\begin{equation*}
f_{a}(t)=\frac{2+(1-t)a}{4}\left( 1+f_{a}(t)\right) -\frac{2+(1+t)a}{4}%
\left( 1-f_{a}(t)\right) =\frac{2+a}{2}f_{a}(t)-\frac{a}{2}t,
\end{equation*}%
that is $f_{a}(t)=t$. On the other hand, summation of %
\eqref{eqn:functional_equation1} and \eqref{eqn:functional_equation2} leads
to
\begin{equation*}
g_{a}(t)=\frac{at}{2}f_{a}(t)+\frac{2-a}{2}=\frac{2-a+at^{2}}{2}.
\end{equation*}%
Therefore, we obtain
\begin{equation*}
H_{a}(t)=\frac{f_{a}(t)+g_{a}(t)}{2}=\frac{(1+t)(2+at-a)}{4},
\end{equation*}%
as desired. $\square $

\end{document}